\documentclass[11pt]{article}
\usepackage{amsmath, latexsym, amsfonts, amssymb, amsthm, amscd}
\usepackage{color}
\usepackage{dsfont}
\usepackage{ulem}

\numberwithin{equation}{section}

\newtheorem{theorem}{Theorem}[section]

\newtheorem{proposition}[theorem]{Proposition}
\newtheorem{lemma}[theorem]{Lemma}

\newcommand{\ud}{\mathrm{d}} 
\renewcommand{\P}{\mathbb{P}}
\newcommand{\R}{\mathbb{R}}

\newcommand{\E}{\mathbb{E}}

\title{
\textbf{Extinction  and  coming down  from infinity of CB-processes with competition in a L\'evy  environment}}

\author{ H. Leman\footnote{ {\sc Centro de Investigaci\'on en Matem\'aticas A.C. Calle Jalisco s/n. 36240 Guanajuato, M\'exico.} E-mail: helene.leman@cimat.mx. Corresponding author}\,\, and J.C. Pardo\footnote{ {\sc Centro de Investigaci\'on en Matem\'aticas A.C. Calle Jalisco s/n. 36240 Guanajuato, M\'exico.} E-mail: jcpardo@cimat.mx}\\}

\begin{document}

\maketitle
\begin{abstract}
In this note, we are interested on  the event of extinction  and the  property of  coming  down from infinity of continuous state  branching (or CB for short) processes with competition  in a L\'evy environment whose branching mechanism satisfies the  so-called Grey's condition.  In particular, we deduce, under  the assumption that the L\'evy  environment does not drift towards infinity,  that for any starting point the process becomes  extinct in finite time a.s.  Moreover if we  impose an integrability condition on the competition mechanism, then the process comes down from infinity regardless the long term behaviour of the environment. 

	\bigskip 
	
		\bigskip
		
\end{abstract}

\section{Introduction and main results.}

A continuous state branching process (or CB-process for short)  is a $[0,\infty]$-valued strong Markov process $Y=(Y_t,t\geq 0)$ with c\'adl\'ag paths satisfying the branching property, that is to say, for all $\theta\geq 0$ and $x,y\geq 0$,
	\[
	\mathbb{E}_{x+y}\Big[e^{-\theta Y_t}\Big]=\mathbb{E}_{x}\Big[e^{-\theta Y_t}\Big]\mathbb{E}_{y}\Big[e^{-\theta Y_t}\Big].
	\]
	This model arises as the scaling limit of Bienaym\'e-Galton-Watson (or BGW for short) processes; where individuals behave independently one from each other and each individual gives birth to a random number of offspring, with the same offspring distribution (see for instance Grimvall \cite{gri}).
	Moreover, its law  is completely characterized by the latter identity, i.e.
	\begin{equation*}
	\mathbb{E}_{x}\Big[e^{-\theta Y_t}\Big]=e^{-xu_t(\lambda)},\qquad  t\geq 0,
	\end{equation*}
	where $u$ is a differentiable  function in $t$ satisfying
	\begin{equation}
	\frac{\partial u_{t}(\lambda)}{\partial t}=-\psi(u_{t}(\lambda)), \qquad u_{0}(\lambda)=\lambda,
	\label{DEut}
	\end{equation}
	and $\psi$ satisfies 
	\begin{equation}\label{branching mechanism}
	\psi(\lambda)=-b\lambda+\gamma^2
	\lambda^2+\int_{(0,\infty)}\big(e^{-\lambda x}-1+\lambda x\big)\mu(\ud x),\qquad \lambda \geq 0,
	\end{equation}
	where $a,\gamma\in \mathbb{R}$ and $\mu$ is a
	measure concentrated on $(0,\infty)$ such that
	\begin{equation}\label{mugral}
\int_{(0, \infty)}(z\land z^2)\mu(\ud z)<\infty.
\end{equation}
The function $\psi$  is known as the branching mechanism of $Y$. A process in this class can also be defined as the unique non-negative strong solution of the following stochastic differential equation (SDE for short) 
	\begin{equation*}
	\begin{split}
	Y_t=Y_0+b\int_0^t Y_s &\ud s+\int_0^t \sqrt{2\gamma^2 Y_s}\ud B_s +\int_0^t\int_{(0,\infty)}\int_0^{Y_{s-}}z\widetilde{N}(\ud s,\ud z,\ud u),
	\end{split}
	\end{equation*}
	where $B=(B_t,t\geq 0)$ is a standard Brownian motion, $N(\ud s,\ud z,\ud u)$ is a Poisson random measure independent of $B$, with intensity $\ud s\mu(\ud z)\ud u$  and $\widetilde{N}$ is the compensated measure of $N$. { Solutions for this type of SDE has been studied  before, see for instance  Dawson and Li \cite{DL} and Caballero et al. \cite{CLU} and the references therein. }

A natural way to extend and make this model more realistic is by considering competition pressure. Such type of models has been considered  recently by several authors,  under the name of CB-processes  with competition; see for instance  Ba and Pardoux \cite{BPa}, Berestycki et al. \cite{ BFF}, Foucart \cite{Fou}, Lambert \cite{Lambert2005},  Ma \cite{Ma2015} and Pardoux \cite{Pardoux} and the references therein. A well known example of this family of processes is the so called logistic Feller diffusion which can  be constructed as scaling limits of BGW-processes with competition, see for instance  Lambert \cite{Lambert2005}.
 
CB-processes with competition can also be defined as the unique strong solution of an SDE and they are determined by two  components; a branching mechanism $\psi$ and a competition mechanism $g$. The  competition mechanism $g$ is a non-decreasing continuous function  on $[0,\infty)$ with $g(0)=0$. According to Ma \cite{Ma2015} (see also Berestycki et al. \cite{BFF}) a CB-process  with  competition  $Y=(Y_t, t\ge 0)$ can be defined as the unique strong solution of the following  SDE 
\begin{equation}\label{cbpcom}
Y_t=Y_0+b\int_0^t Y_s \ud s-\int_0^tg(Y_s)\ud s+\int_0^t\sqrt{2\gamma^2Y_s}\ud B^{(b)}_s+\int_0^t\int_{(0,\infty)}\int_0^{Y_{s-}} z\widetilde{N}^{(b)}(\ud s, \ud z, \ud u),
\end{equation}
where $B^{(b)}$ is a standard Brownian motion and   $N^{(b)}$ is a Poisson random measure  which is defined  on $\mathbb{R}_+^3$, with intensity measure $\ud s \mu(\ud z) \ud u$ such that \eqref{mugral} is satisfied, and $\widetilde{N}^{(b)}$ denotes its compensated version. 

Lambert \cite{Lambert2005} studied the long term behaviour of the logistic case i.e. $g(x)=cx^2,$ for $x\ge 0$ and $c>0$, using a Lamperti-type  representation (random time change) where the driven process turns out to be a generalised  Ornstein-Uhlenbeck process driven by a spectrally positive L\'evy process, here denoted by $X$, satisfying a logarithmic moment condition.   To be more precise when the process $X$ is a subordinator  then the associated logistic branching process $Y$  may converge to a specified distribution or  to $0$ in probability (see Theorem 3.4 in \cite{Lambert2005}). When $X$   is not a subordinator, then the process $Y$ goes to $0$ a.s. Moreover,  the process $Y$ gets extinct in finite time a.s.
accordingly as 
\begin{equation}\label{grey}
\int^\infty \frac{\ud u}{\psi(u)}<\infty,
\end{equation}
which is the so-called Grey's condition.  In \cite{Lambert2005}, under Grey's condition, the  Laplace transform of the extinction time was computed explicitly  and   the law of the process coming down from infinity was also determined. 

More general competition mechanisms were  considered by  Ba and Pardoux \cite{BPa} in the case where the branching mechanism is of the form $\psi(u)=\gamma^2 u^2$, for $u\ge 0$, see also Chapter 8 in the book of Pardoux \cite{Pardoux}. Actually, they allow the competition mechanism $g$ to be a continuous function and not necessarily monotone and provided a necessary and sufficient condition for the process to become extinct. In this setting, if   $g$ is  negative and non-increasing then the competition mechanism can be  interpreted as cooperation   in the sense of Gonzalez-Casanova et al. \cite{GPP}.

Branching processes in random environment (BPREs) were first introduced and studied by Smith and Wilkinson \cite{smith} and since then they have attracted considerable interest  (see for instance \cite{afa12,bo10} and the references therein). BPREs are interesting since they are more realistic models compared with classical branching processes and, from the mathematical point of view, they have new properties such as another phase transition in the subcritical regime. 
	Scaling limits in the finite variance case were conjectured by Keiding \cite{kei} who introduced Feller diffusions in random environment. This conjecture was proved by Kurtz \cite{Kurtz} and more recently by Bansaye and Simatos \cite{basim} in  more general cases.  The continuous state version, with an environment driven by a L\'evy process, was recently introduced independently by  He et al. \cite{he2016continuous}  and Palau and Pardo \cite{PP}, as the unique strong solution of an SDE (see below) under the name of CB-processes in L\'evy environment. 
	
In this paper, we are interested in  CB-processes with  competition in a L\'evy  environment which are defined  as the unique strong solution of the following  SDE
\begin{equation}\label{SDE_Levy}
\begin{split}
Z_t&=Z_0+b\int_0^tZ_s \ud s-\int_0^tg(Z_s)\ud s+\int_0^t\sqrt{2\gamma^2Z_s}\ud B^{(b)}_s\\
&\hspace{4cm}+\int_0^t\int_{(0,\infty)} \int_0^{Z_{s-}}z\widetilde{N}^{(b)}(\ud s, \ud z, \ud u)+\int_0^t Z_{s-}\ud  S_s,
\end{split}
\end{equation}
where $g$ is a non-decreasing continuous function  on $[0,\infty)$ with $g(0)=0$,  $B^{(b)}$ and  $N^{(b)}$ are defined as before  and $S$ is a L\'evy process independent of $B^{(b)}$ and  $N^{(b)}$ which can be written as follows
\begin{equation}\label{defS}
S_t={\tt d} t+ \sigma B^{(e)}_t +\int_0^t \int_{(-1,1)^c} (e^z-1)N^{(e)}({\ud s},{\ud z})+\int_0^t \int_{(-1,1)} (e^z-1)\widetilde{N}^{(e)}({\ud s},{\ud z}),
\end{equation}
 with ${\tt d}\in \mathbb{R},  \sigma\geq 0$,  $B^{(e)}=(B^{(e)}_t,t\ge 0 )$ is a standard Brownian motion and  $N^{(e)}$ is a Poisson random measure taking values on $\R_+\times \R$  with intensity ${\ud s} \pi({\ud z})$ satisfying
 \begin{equation}\label{defpi}
 \int_{\mathbb{R}\setminus\{0\}} (1\land z^2) \pi(\ud z)< \infty.
 \end{equation}
In particular,  our aim is to determine under which conditions  such  family of processes  becomes extinct  (with positive probability or almost surely) and comes down from infinity. 
  
 For  our purposes, we also introduce the auxiliary L\'evy process  which is a modification of  $S$,
\[
K_t=\mathbf{m} t+ \sigma B^{(e)}_t +\int_0^t \int_{(-1,1)^c} zN^{(e)}({\ud s},{\ud z})+\int_0^t \int_{(-1,1)} z\widetilde{N}^{(e)}({\ud s},{\ud z}), \qquad t\ge 0,
\]
where
\[
\mathbf{m}= b+{\tt d}-\frac{\sigma^2}{2}-\int_{(-1,1)} (e^z-1-z)\pi(\ud z).
\]
It is important to note that the drift term $\mathbf{m}$ of the process $K$ provides the interaction between the demographic and environmental parameters. 

We denote by $\mathbb{P}_x$ the law of $Z$ starting from $x>0$, and we define by $T_0=\inf\{t\geq 0, Z_t=0\}$  the first hitting time to $0$ of $Z$,
with the convention that $\inf\{\emptyset\}=\infty$.  In the case without competition i.e. $g\equiv 0$,  He et al. \cite{he2016continuous} proved that the so-called Grey's condition \eqref{grey} is  necessary and sufficient  for CB-processes in a L\'evy  environment 
to become extinct  with positive probability (see Theorem 4.1 in   \cite{he2016continuous}). Moreover, if the auxiliary process $K$ does not drift to $\infty$ or equivalently
\begin{equation}
\label{eq_driftinfty}
\liminf_{t\to \infty}K_t=-\infty,
\end{equation}
and  Grey's condition \eqref{grey} holds, then its associated CB-process in a L\'evy environment becomes extinct at finite time a.s., see Corollary 4.4 in \cite{he2016continuous}. It is important to note that when condition \eqref{eq_driftinfty} is fullfilled, the associated CB-process in a  L\'evy environment is critical or subcritical. 

We also point out that under the  assumption \eqref{mugral}, the CB-process in a L\'evy environment does not explode. The proof of this claim follows exactly the same arguments of Proposition 1 in \cite{PP1} where the authors consider the specific case of Brownian environment. Indeed, in their arguments the environment does not play any role, when assumption \eqref{mugral} is fulfilled the non-explosion  only depends on the branching mechanism.

Our first result, which follows from a comparison criteria for CB-processes with  competition in a L\'evy  environment (see Lemma \ref{lemma_comp} below), gives a necessary condition under which they become extinct. Before stating it, we introduce the CB-process in a L\'evy environment  $ Z^\sharp=(Z^\sharp_t, t\ge 0)$ as the unique strong solution of \eqref{SDE_Levy} but with $g\equiv 0$. For simplicity, we denote its  law  starting from $x> 0$ by $\mathbb{P}^\sharp_x$. 

\begin{proposition}\label{comparison}
Assume that the L\'evy measure $\mu$ associated to the branching mechanism $\psi$ satisfies \eqref{mugral}. For $y\ge x\ge 0$, we have that $(Z, \mathbb{P}_x)$  is stochastically dominated by $(Z, \mathbb{P}_y)$, i.e.
\[
\mathbb{P}\Big(Z_t(x)\le Z_t(y)\,\, \textrm{ for all }\,\, t\ge 0\Big)=1,
\]
where $Z(v)$ denotes $(Z, \mathbb{P}_v)$, under $\mathbb{P}$.
Moreover, the process  $(Z, \mathbb{P}_x)$ is stochastically dominated by  $( Z^\sharp, \mathbb{P}^\sharp_y)$ and, in particular, if the branching mechanism $\psi$ satisfies Grey's condition \eqref{grey},  then $(Z, \mathbb{P}_x)$ becomes extinct  with positive probability and its semigroup is strong Feller. Furthermore  if $K$ does not drift to $\infty$ or equivalently satisfies \eqref{eq_driftinfty}
then $(Z, \mathbb{P}_x)$ becomes extinct at finite time a.s.
\end{proposition}
For the sequel, we always assume that Grey's condition is fulfilled. In other words,  the CB-processes with  competition in a L\'evy  environment that we are considering here  become extinct  with positive probability and are strong Feller.\\

We now state our main result which provides a sufficient condition on the competition mechanism for CB-processes with  competition in a L\'evy  environment  to become extinct a.s. even for favorable environments, i.e. when $K$ drifts to $+\infty$.  In other words, the condition on the competition parameter is so strong that the process become extinct regardless the long term behaviour of the environment. To this aim, we assume  the following integral condition on the competition mechanism $g$: assume that there exists $z_0 >0$ such that $g(z_0)>0$ and 
\begin{equation}\label{H2} 
  \int_{z_0}^{\infty} \frac{\ud y}{g(y)}<\infty.
 \end{equation}
Actually, the above condition
 implies that  the associated  CB-process with competition in a L\'evy random environment {\it comes down from infinity}.  This phenomenon has been observed and studied by several authors in branching processes with interactions, see for instance  Gonz\'alez-Casanova et al. \cite{GPP}, Lambert \cite{Lambert2005}, Li \cite{Li2016}, Li et al. \cite{LiYhangZhou} and Pardoux \cite{Pardoux} and also for stable jump diffusions by D\"oring and Kyprianuo \cite{DK} and some jump diffusions by Bansaye  \cite{Ban}. Formally,  we define the property of {\it coming down from infinity} in the sense that $\infty$ is a {\it continuous entrance point}, i.e.
\[
\lim_{M\to \infty}\lim_{x\to \infty}\mathbb{P}_x(T_M<t)=1 \qquad \textrm{for all} \quad t>0, 
\]
where $T_M=\inf\{t\ge 0: Z_t\le M\}$ and 
the original process can be extended into a Feller process on  $[0,\infty]$  (see for instance Theorem 20.13 in Kallenberg~\cite{Kallenberg} for the diffusion case or Definition 2.2 for Feller processes in \cite{DK}).

\begin{theorem}
\label{theo_meanfinite}
Assume that the L\'evy measure $\mu$ associated to the branching mechanism satisfies \eqref{mugral}.  If Grey's condition \eqref{grey} and  \eqref{H2} hold,  then 
\[
\sup_{x>0}\mathbb{E}_x[T_0]<\infty,
\]
  the boundary point  $\infty$ is a continuous entrance point and  the process  $Z$ comes down from infinity.

\end{theorem}

The previous result can be applied to the particular case when the competition mechanism is logistic (i.e. $g(x)=cx^2$) and  the random environment is driven by a Brownian motion. Actually,  further explicit computations can be carried out for the Laplace transform of the extinction time under $\mathbb{P}_\infty$, the law of the process starting from $\infty$, as it is done in  Leman and Pardo \cite{leman2018extinction1} where the case of the logistic branching process in a Brownian environment  is presented.

The remainder of this note is devoted to the proofs.

\section{Proofs}
In order to prove Proposition \ref{comparison}, we introduce the following stochastic processes as unique strong solutions of the  SDE's. For $i=1,2$, we let 
\[
\begin{split}
Z^{(i)}_t&=Z^{(i)}_0+\int_0^tg_i(Z^{(i)}_s)\ud s+\int_0^t\sqrt{2\gamma^2Z^{(i)}_s}\ud B^{(b)}_s+\int_0^t\int_{(0,\infty)} \int_0^{Z^{(i)}_{s-}}z\widetilde{N}^{(b)}(\ud s, \ud z, \ud u)+\int_0^t Z^{(i)}_{s-}\ud  S^{(i)}_s, 
\end{split}
\]
where
\begin{equation*}
S^{(i)}_t={\tt d} t+ \sigma B^{(e)}_t +\int_0^t \int_{(-1,1)^c} b_i(z)N^{(e)}({\ud s},{\ud z})+\int_0^t \int_{(-1,1)} (e^z-1)\widetilde{N}^{(e)}({\ud s},{\ud z}),
\end{equation*}
with $g_1(z)\ge g_2(z)$, for $z\ge 0$,  and  $b_1(z)\ge b_2(z)$ for $z\in \mathbb{R}$ such that, for $i=1,2$
\[
b_i(z)+1\ge 0, \qquad \textrm{for}\quad z\in \mathbb{R}.
\]
We also  assume that for each $m\ge 0$, there is a non-decreasing concave function $z\mapsto r_m(z)$ on $\mathbb{R}_+$ satisfying 
$\int_{0+} r_m(z)\ud z=\infty$ and
\begin{equation}\label{convexr}
|g_i(x)-g_i(y)|+{\tt d}|x-y|+|x-y|\int_{(-1,1)^c}\Big(|b_i(z)|\land m\Big)\pi(\ud z)\le r_m(|x-y|), \qquad \textrm{for}\quad i=1,2,
\end{equation}
for every $0\le x, y\le m$. According to Proposition 1 in Palau and Pardo,  the previous SDE's possess  unique positive strong solutions that we denote by $Z^{(i)}$ for $i=1,2$.  

Our next result can be deduced using similar arguments as those used in the proof of  Theorem 2.2 in \cite{DL}. For simplicity on exposition we provide its complete proof.

\begin{lemma}\label{lemma_comp}
If $Z^{(1)}_0\ge Z^{(2)}_0$, a.s. then 
\[
\mathbb{P}\Big(Z^{(2)}_t\le Z^{(1)}_t\,\, \textrm{ for all }\,\, t\ge 0\Big)=1.
\]
\end{lemma}
\begin{proof} 
Let $\tau_m=\inf\{t\ge 0: Z^{(1)}_s\ge m \textrm{ or } Z^{(2)}_s\ge m\}$ for $m\ge 1$.  According to the proof of Proposition 1 in Palau and Pardo \cite{PP}, for $i=1, 2$, we have $Z^{(i)}_t=Z^{(i, m)}_t$ for $t<\tau_m$, where $Z^{(i, m)}$ is the unique strong solution to
\[
\begin{split}
Z^{(i, m)}_t&=Z^{(i)}_0+\int_0^tg_i(Z^{(i, m)}_s\land m)\ud s+\int_0^t\sqrt{2\gamma^2Z^{(i, m)}_s\land m}\,\,\ud B^{(b)}_s\\
&+\int_0^t\int_{(0,\infty)} \int_0^{Z^{(i, m)}_{s-}\land m}(z\land m)\widetilde{N}^{(b)}(\ud s, \ud z, \ud u)+\int_0^t \Big(Z^{(i, m)}_{s-}\land m\Big)\ud  S^{(i, m)}_s, 
\end{split}
\]
where
\begin{equation*}
S^{(i)}_t={\tt d} t+ \sigma B^{(e)}_t +\int_0^t \int_{(-1,1)^c} \Big(b_i(z)\land m\Big)N^{(e)}({\ud s},{\ud z})+\int_0^t \int_{(-1,1)} \Big((e^z-1)\land m\Big)\widetilde{N}^{(e)}({\ud s},{\ud z}).
\end{equation*}
In other words for $m\ge 1$, we have 
\[
\mathbb{P}\left(Z^{(1)}_t\ge Z^{(2)}_t, \textrm{ for all} \,\, t<\tau_m\right)=\mathbb{P}\left(Z^{(1, m)}_t\ge Z^{(2, m)}_t, \textrm{ for all } \,\, t<\tau_m\right).
\]
Then, a direct application of Theorem 2.2 in \cite{DL} implies that the latter probability equals one. This ends the proof of Lemma~\ref{lemma_comp}.
 \end{proof}
 
\begin{proof}[Proof of Proposition \ref{comparison}] The first statement follows directly from Lemma \ref{lemma_comp} by taking 
\[
g_1(z)=g_2(z)=({\tt d}+b)z-g(z)\quad \textrm{ for } z\ge 0\qquad  \textrm{and}\qquad b_1(z)=b_2(z)=e^{z}-1 \quad \textrm{ for } z\in \mathbb{R}.
\]
For the second statement, we recall that the competition mechanism $g$ is positive and non-decreasing implying that we can take $g_1(z)=({\tt d}+b)z$, $g_2(z)=({\tt d}+b)z-g(z)$ and $b_1(z)=b_2(z)=e^{z}-1$. Again from Lemma \ref{lemma_comp}, we deduce  that the process  $(Z, \mathbb{P}_x)$ is stochastically dominated by  $( Z^\sharp, \mathbb{P}^\sharp_y)$ for $y\ge x$. 
In other words,  from  Theorem 4.1 and Corollary 4.4 in   \cite{he2016continuous}, we deduce that  $(Z, \mathbb{P}_x)$  becomes extinct  with positive probability and that  if $K$ does not drift to $\infty$ or equivalently satisfies \eqref{eq_driftinfty}
then the process becomes extinct at finite time a.s.

In order to conclude our proof, it remains to deduce  that the process $(Z, \mathbb{P}_x)$ is strong Feller under Grey's condition \eqref{grey}. To do so, we use a similar argument as in Theorem 4.5 \cite{he2016continuous}.  We introduce another formulation  of  CB-processes with competition in a L\'evy environment for all initial values.  From Theorem III.6 in El Karoui and M\'el\'eard \cite{EKM}, on an extension of the original probability space we can define  $W(\ud s, \ud u)$  a time-space Gaussian white noise on $(0,\infty)^2$ with intensity $\ud s \ud u$ such that \eqref{SDE_Levy} may be rewritten as follows
\[
\begin{split}
Z_t&=x+b\int_0^t Z_s\ud s-\int_0^t g(Z_s) \ud s +\sqrt{2{\gamma^2}}\int_0^t \int_{0}^{Z_s}W(\ud s, \ud u)\\
&\hspace{3cm}+\int_0^t\int_{(0,\infty)} \int_0^{Z_{s-}}z\widetilde{N}^{(b)}(\ud s, \ud z, \ud u)+\int_0^t Z_{s-}\ud  S_s.
\end{split}
\]
According to Proposition 1 in Palau and Pardo \cite{PP} (see also  the proof of Theorem 4.5 of He et al. \cite{he2016continuous} for the case $g\equiv 0$), for each $x\ge 0$, there is a unique strong solution of the previous SDE
that we denote by $(\Gamma^{(x)}_t, t\ge0)$ with $\Gamma^{(x)}_0=x$, which is also a Markov process with the same transition semigroup as $(Z_t, \mathbb{P}_x)$, i.e.
\[
{\tt P}_tf(z):=\mathbb{E}_z[f(Z_t)]=\mathbb{E}\left[f\left(\Gamma^{(z)}_t\right)\right].
\]
Similar arguments as those used in the proof of Lemma \ref{lemma_comp} allow us to conclude that, for $y\ge 0$, we have
\[
\mathbb{P}\left(\Gamma^{(y)}_t\ge \Gamma^{(x)}_t, \quad \textrm{for all} \quad t\ge 0\right)=1.
\]
We claim that the process $\Gamma_t^{(x,y)}:=\Gamma^{(y)}_t-\Gamma^{(x)}_t$, for $ t\ge 0$, satisfies a similar equation. Indeed 
\[
\begin{split}
\Gamma^{(x,y)}_t&:= y-x+b\int_0^t \Gamma^{(x,y)}_s\ud s-\int_0^t G_{\Gamma^{(x)}_t}(\Gamma^{(x,y)}_s) \ud s +\sqrt{2\gamma^2} \int_0^t \int_0^{\Gamma^{(x,y)}_s}W^\prime(\ud s, \ud u)\\
&\hspace{3cm}+\int_0^t\int_{(0,\infty)} \int_0^{\Gamma^{(x,y)}_{s-}}z\widetilde{N}^{(b)\prime}(\ud s, \ud z, \ud u)+\int_0^t \Gamma^{(x,y)}_{s-}\ud  S_s, 
\end{split}
\]
where $G_x(z)=g(x+z)-g(x)$  is a non-decreasing continuous and positive function on $[0,\infty)$,
\[
W^\prime(\ud s, \ud u)=W(\ud s, \ud u +\Gamma^{(x)}_s),
\]
and 
\[
\widetilde{N}^{(b)\prime}_x(\ud s, \ud z, \ud u)=\widetilde{N}^{(b)}_x(\ud s, \ud z, \ud u+\Gamma^{(x)}_s)
\]
which are respectively a Gaussian white noise with intensity $\ud s \ud u$ and a Poisson random measure with intensity $\ud s\mu(\ud z) \ud u$. Using again  the version of Lemma \ref{lemma_comp} for the SDE with Gaussian white noise, we deduce  
\[
\mathbb{P}\left(\Gamma^{(x,y),\sharp}_t\ge \Gamma^{(x,y)}_t, \quad \textrm{for all} \quad t\ge 0\right)=1,
\]
where the process $\Gamma^{(x,y), \sharp}$ is the unique strong solution of 
\[
\begin{split}
\Gamma^{(x,y), \sharp}_t&:= y-x+b\int_0^t \Gamma^{(x,y), \sharp}_s\ud s +\sqrt{2\gamma^2} \int_0^t \int_0^{\Gamma^{(x,y), \sharp}_s}W^\prime(\ud s, \ud u)\\
&\hspace{3cm}+\int_0^t\int_{(0,\infty)} \int_0^{\Gamma^{(x,y), \sharp}}z\widetilde{N}^{(b)\prime}(\ud s, \ud z, \ud u)+\int_0^t \Gamma^{(x,y), \sharp}_{s-}\ud  S_s, 
\end{split}
\]
which is equivalent to $( Z^\sharp, \mathbb{P}^\sharp_{y-x})$. In other words, the process $\Gamma^{(x,y)}$ hits $0$ a.s., and  from its dynamics we observe that  $0$ is an absorbing boundary.

Now, let $f\in \mathcal{B}_b(\mathbb{R})$, the space of bounded measurable functions, and let 
\[
T^{(x,y)}=\inf\{t\ge 0: \Gamma^{(y)}_t=\Gamma^{(x)}_t\}.
\] 
Then $\Gamma^{(y)}_t=\Gamma^{(x)}_t$ for $t\ge T^{(x,y)}$ and observe,
\[
\begin{split}
\Big|{\tt P}_t f(y)-{\tt P}_t f(x)\Big|&\le \mathbb{E}\left[\Big|f(\Gamma^{(y)}_t)-f(\Gamma^{(x)}_t)\Big|\mathbf{1}_{\{t<T^{(x,y)}\}}\right]\\
&\le2 \|f\|_{\infty} \mathbb{P}(t<T^{(x,y)})\\
&\le 2\|f\|_{\infty} \mathbb{P}^\sharp_{y-x}(T_0>t)\\
&\le 2\|f\|_{\infty} \mathbb{E}\left[1-e^{-(y-x)v_t(0,\infty, \overline{K})}\right],
\end{split}
\]
where $\|\cdot\|_{\infty}$ denotes the supremum norm and $v_t(0,\infty,\overline{K})$ is a functional of the environment such that
\[
\mathbb{P}^\sharp_{y}(Z_t=0|\overline{K})=e^{-yv_t(0,\infty,\overline{K})}>0,
\]
see for instance Theorems 4.1 and  4.3 in He et al. \cite{he2016continuous}. In other words, $|{\tt P}_tf(y)-{\tt P}_tf(x)|$ goes to $0$ as $|x-y|$ goes to $0$, implying that ${\tt P}_t f$ is a continuous function on $[0, \infty)$. The proof is now complete. 
\end{proof}

We now prove Theorem \ref{theo_meanfinite}. 

\begin{proof}[Proof of Theorem \ref{theo_meanfinite}]  First of all, from Lemma~\ref{lemma_comp}, it is enough to prove our result for a process with a random environment which has no downward jumps larger than $1-e^{-1}$. Hence, we assume in all this proof that $Z$ is solution to~\eqref{SDE_Levy} with
\begin{equation}\label{SDE_SLevy}
S_t={\tt d} t+ \sigma B^{(e)}_t +\int_0^t \int_{(1,\infty)} (e^z-1)N^{(e)}({\ud s},{\ud z})+\int_0^t \int_{(-1,1)} (e^z-1)\widetilde{N}^{(e)}({\ud s},{\ud z}),
\end{equation}

Our first step is to prove that the expectation of the extinction time of the process is finite.
Recall that  $T_M$  denotes  the first passage time for the process $Z$ below a level $M>0$, i.e.
$T_M:=\inf \{t \geq 0, Z_t\leq M\}.$
 As we will see below, the finiteness of the first moment of such random times will be useful for deducing  our result. Hence,  we first show that there exist $M>0$ such that 
\begin{equation}
 \label{firststep}
  \sup_{x\geq 0} \E_x\big[ T_M\big] = \sup_{x\geq M} \E_x\big[ T_M\big] <\infty.
\end{equation}
In order to deduce \eqref{firststep}, we use similar arguments  as those  used in Le \cite{le2014processus}.
With this goal in mind, we observe from Assumption \eqref{H2} that
\begin{equation}\label{tradH2}
 \lim_{y\to +\infty} \frac{g(y)-\theta y}{y}=\infty,
\end{equation}
for $\theta:=\max\{b+{\tt d},0\}$. In addition from Lemma 2.3 in Le and Pardoux \cite{vi2015height}, we deduce that there exists $a_0 >0$ such that $g(y)-\theta y>0$ for any $y\geq a_0$ and 
\begin{equation}\label{tradH2bis}
  \int_{a_0}^{\infty} \frac{\ud y}{g(y)-\theta y}<\infty.
 \end{equation} 
We then introduce $A>\theta(e-1)$ large enough such that the inequality below holds
\begin{equation}
 \label{ass_A}
 \begin{split}
C(A)&:=1- \left(\frac{\theta(2\gamma^2+\sigma^2)}{2A^2}+\frac{\theta}{A(A-\theta)}\int_{(0,1)}z^2\mu({\ud z})+\frac{1}{A} \left(\int_{(1,\infty)}z\mu({\ud z})+\overline{\pi}(1) \right)\right .\\
&\hspace{6cm}\left . +\left(\frac{\theta}{A^2}+\frac{\theta}{A(A-\theta(e^1-1))}\right)\int_{(-1, 1)} z^2\pi(\ud z)\right)>0,
\end{split}
\end{equation}
where $\overline{\pi}(x)=\pi((x,\infty))$, $x\ge0$. From \eqref{tradH2} and \eqref{tradH2bis}, it is clear that  there exists a constant $M>(a_0+1)e$ such that 
\begin{equation}
\label{ass_M}
\int_{Me^{-1}}^{\infty}\frac{\ud w}{g(w)-\theta w}\leq \frac{1}{A} \qquad \text{and} \qquad g(y)-\theta y\geq Ay \geq A, \qquad \text{ for all } \quad y\geq Me^{-1}.
\end{equation}
 Such constant $M$ will be our threshold. 
 For our purposes, we define the function $G$ in $C^2(\R)$ as follows
 \[
 G(y)=\left\{ \begin{aligned}
 &\int_{a_0}^y\frac{\ud w}{g(w)-\theta w} &&\text{ if } y \geq a_0+1,\\
& 0 && \text{ if } y\leq a_0,
 \end{aligned}
 \right.
 \]
and such that $G$ is non-negative  and non-decreasing. Thus  applying It\^o's formula to $G(Z_{t\land T_M})$, we find
\begin{equation}
\label{itoformula}
\begin{split}
G(Z_{t \wedge T_M})- G(Z_0)&= -t \wedge T_M - \int_0^{t\wedge T_M} \frac{g'(Z_s)-\theta }{(g(Z_s)-\theta Z_s)^2}\left( \gamma^2 Z_s+\frac{\sigma^2}{2}Z_s^2\right){\ud s} \\
& +\int_0^{t\wedge T_M} \frac{\sqrt{2\gamma^2 Z_s}}{g(Z_s)-\theta Z_s}\ud B^{(b)}_s+ \int_0^{t\wedge T_M} \frac{\sigma{Z_s}}{g(Z_s)-\theta Z_s} \ud B^{(e)}_s \\
&+\int_0^{t \wedge T_M}\int_{(0,\infty)} Z_s\left(G(Z_{s}+z)-G(Z_{s})-\frac{z}{g(Z_{s})-\theta Z_s}\right) \mu({\ud z}){\ud s}\\
& +\int_0^{t \wedge T_M}\int_{(0,\infty)}\int_0^{Z_{s-}} [G(Z_{s-}+z)-G(Z_{s-})]
\widetilde{N}^{(b)}({\ud s},{\ud z},{\ud u})\\
& +\int_0^{t \wedge T_M}\int_{(1,\infty)} [G(e^z Z_{s-})-G(Z_{s-})]{N}^{(e)}({\ud s},{\ud z})\\
& +\int_0^{t \wedge T_M}\int_{(-1,1)} [G(e^z Z_{s-})-G(Z_{s-})]\widetilde{N}^{(e)}({\ud s},{\ud z})\\
&+\int_0^{t \wedge T_M}\int_{(-1,1)} \left(G(e^z Z_{s})-G(Z_{s})-\frac{(e^z-1)Z_s}{g(Z_s)-\theta Z_s}\right) \pi(\ud z){\ud s}.\\
\end{split}
\end{equation}
Firstly, note that $Z_{s} \geq Me^{-1} >a_0+1$ for any $s\leq t\land T_M$, in other words, we have an explicit formula for $G(Z_s)$. Next, we take expectations under the assumption that the process $Z$ starts at $x\ge M$, in both sides of the previous identity and we study separately each  term of the right-hand side. Our aim is to show that each expectation can be bounded from above using $\E_x[t\wedge T_M]$. For simplicity, we enumerate the lines in order of appearance.

(1) For the first integral of the right hand side of \eqref{itoformula}, we recall that $Z_s \geq Me^{-1}$  for  $s \leq t \wedge T_M$, that $g$ is non-decreasing and we use the second formula in ~\eqref{ass_M} to  deduce 
\begin{equation}
\label{maxterm2}
\E_x\left[ \int_0^{t\wedge T_M} \frac{\theta-g'(Z_s)}{(g(Z_s)-\theta Z_s)^2}\left( \gamma^2 Z_s+\frac{\sigma^2}{2}Z_s^2\right){\ud s} \right] \leq \frac{\theta(2\gamma^2+\sigma^2)}{2A^2} \E_x\Big[t \wedge T_M\Big]. 
\end{equation}

(2) Studying the quadratic variation of both continuous local martingales of the second line of the right-hand side of \eqref{itoformula} together with the second formula in ~\eqref{ass_M}, we observe that  both processes  are real martingales. Therefore their expectations are equal to $0$.

 (3) We study the integral that appears in the third line  in \eqref{itoformula} by separating $(0,\infty)$  into two parts $(0,1]$ and $(1,\infty)$. We first deal with the integral restricted to $(0,1)$. Since $g$ is non-decreasing, we bound $G(Z_{s}+z)-G(Z_{s})$ from above by $z(g(Z_s)-\theta(Z_s+z))^{-1}$. In addition with the second formula in \eqref{ass_M}, we obtain the following upper bound
\begin{equation}
\label{maxterm3}
\begin{aligned}
\E_x&\left[ \int_0^{t \wedge T_M} \int_{(0,1)}Z_{s} \left[G(Z_{s}+z)-G(Z_{s})-\frac{z}{g(Z_{s})-\theta Z_{s}}\right] \mu({\ud z}){\ud s} \right] \\
&\hspace{2cm} \leq \E_x \left[ \int_0^{t \wedge T_M}\int_{(0,1)}\frac{\theta z^2 Z_{s}}{(g(Z_{s})-\theta (Z_{s}+ z))(g(Z_{s})-\theta Z_{s})}  \mu({\ud z}){\ud s} \right]\\
&\hspace{2cm}\leq \E_x\Big[t \wedge T_M\Big] \frac{\theta}{A(A-\theta)} \int_{(0,1)} z^2\mu({\ud z}) <\infty .
\end{aligned}
\end{equation}
Concerning the integral restricted to $(1,\infty)$, we drop the last term, which is negative, and we use again the second formula in~\eqref{ass_M} to bound $G(Z_s+z)-G(Z_s)$ and find
\begin{equation}
\label{maxterm3bis}
\begin{aligned}
\E_x&\left[ \int_0^{t \wedge T_M} \int_{(1,\infty)}Z_{s} \left[G(Z_{s}+z)-G(Z_{s})-\frac{z}{g(Z_{s})-\theta Z_{s}}\right] \mu({\ud z}){\ud s} \right] \\
&\hspace{1cm} \leq  \E_x  \left[\int_0^{t \wedge T_M}\int_{(1,\infty)} \int_0^z \frac{Z_{s}}{A(Z_{s}+w)} \ud w \mu({\ud z}) {\ud s}\right]
\leq \frac{\int_{(1,\infty)}z \mu({\ud z})}{A}  \E_x\Big[t \wedge T_M\Big] .
\end{aligned}
\end{equation}

(4) Again, we split the interval $(0,\infty)$ into $(0,1]$ and $(1,\infty)$ and use similar computations  as in part (3)  in order to deduce that the integral restricted to $(0,1]$ is a square integrable martingale and the integral restricted to $(1, \infty)$ is a martingale. In other words, we manipulate
\[\E_x\left[ \int_0^{t \wedge T_M}\int_{(0,\infty)}Z_{s} f(G(Z_{s}+z)-G(Z_{s}),z) \mu({\ud z}){\ud s} \right],\] 
with $f(x,z)=x^2\mathbf{1}_{(0,1]}(z)$  and $f(x,z)=|x|\mathbf{1}_{(1,\infty)}(z)$ respectively. Their expectations are thus $0$.

(5) Similarly, using Fubini's Theorem and the first inequality in \eqref{ass_M}, we deduce 
\begin{equation*}
\begin{aligned}
\E_x\Bigg[ \int_0^{t \wedge T_M}\int_{(1,\infty)} |G(e^z Z_{s-})-G(Z_{s-})|\pi(\ud z) \ud s\Bigg] 
& \leq \E_x\left[ \int_0^{t \wedge T_M}  \int_{Z_{s}}^{\infty} \frac{\ud w}{g(w)-\theta w} \left( \int_{1} ^{\infty} \pi(\ud z) \right)\ud s \right]\\
& \leq t  \overline{\pi}(1) \left(\int_{Me^{-1}}^{\infty} \frac{\ud w}{g(w)-\theta w}\right)\leq \frac{t}{A} \overline{\pi}(1).
\end{aligned}
\end{equation*}
In other words, the stochastic integral of the fifth term can be written as the sum of a martingale and a finite variation process. Moreover its expectation is bounded from above by $\E[t\wedge T_M]  {\overline{\pi}(1)}/{A}$.

(6) Since $g(w)-\theta w \geq Aw$, we find that the integral term of the sixth line is a square integrable martingale. Indeed, we observe
\[
\begin{split}
\E\left[ \int_0^{t \wedge T_M}\int_{(-1, 1)} \left(\int_{Z_{s}}^{e^z Z_{s}} \frac{\ud w}{g(w)-\theta w}\right)^2 \pi({\ud z}){\ud s} \right]& \leq  \E\left[ \int_0^{t \wedge T_M}\int_{(-1,1)} \left(\int_{Z_{s}}^{e^z Z_{s}} \frac{\ud w}{Aw} \right)^2 \pi({\ud z}){\ud s} \right]\\
& \leq \frac{t}{A^2} \int_{(-1,1)} z^2 \pi({\ud z}) <\infty.
\end{split}
\]
In other words,  its expectation is equal to $0$.

(7) Finally, we study the last line in \eqref{itoformula} by splitting again the integral into two parts, i.e. we split $(-1,1)$ into $(-1,0]$ and $(0,1)$. Thus, using again the second inequality of \eqref{ass_M} and the fact that $A>\theta(e-1)$, we deduce that for any $w \in [0,y(e^z-1)]$, $y\ge 1$ and $z\in (-1,1)$,
 \[
  g(y+w)-\theta (y+w) \geq g(y)-\theta y e^z \geq Ay-\theta (e-1) y  > 0.
 \]
Hence,
\begin{equation}\label{maxterm5a}
\begin{split}
\E_x&\left[ \int_0^{t \wedge T_M} \int_{(0,1)} \left(G(e^z Z_{s})-G(Z_{s})-\frac{(e^z-1)Z_s}{g(Z_s)-\theta Z_s}\right) \pi(\ud z){\ud s} \right] \\
& \hspace{1cm} \leq \E_x\left[ \int_0^{t \wedge T_M} \int_{(0,1)} \left(\frac{(e^z-1)Z_s}{g(Z_s)-\theta Z_s e^z}-\frac{(e^z-1)Z_s}{g(Z_s)-\theta Z_s}\right) \pi(\ud z){\ud s} \right] \\ 
&\hspace{1cm} \leq \E_x \left[ \int_0^{t \wedge T_M}\int_{(0,1)}\frac{\theta (e^z-1)^2 (Z_{s})^2}{(A Z_{s}-\theta(e^1-1) Z_s)(g(Z_{s})-\theta Z_{s})}  \pi({\ud z}){\ud s} \right]\\
&\hspace{1cm}\leq \E_x\Big[t \wedge T_M\Big] \frac{\theta \int_{(0,1)} z^2\pi({\ud z})}{A(A-\theta(e^1-1))} .
\end{split}
\end{equation}
Similarly,  we deal with the second part of the integral and deduce
\begin{equation}\label{maxterm5b}
\begin{split}
\E_x&\left[ \int_0^{t \wedge T_M} \int_{(-1,0)} \left(G(e^z Z_{s})-G(Z_{s})-\frac{(e^z-1)Z_s}{g(Z_s)-\theta Z_s}\right) \pi(\ud z){\ud s} \right] 
\leq \E_x\Big[t \wedge T_M\Big] \frac{\theta \int_{(-1,0)}z^2\pi({\ud z})}{A^2} .
\end{split}
\end{equation}

Thus putting all pieces together (i.e. inequalities  \eqref{maxterm2}, \eqref{maxterm3}, \eqref{maxterm3bis}, \eqref{maxterm5a}, \eqref{maxterm5b} and the bound found in (5) together with \eqref{itoformula}, \eqref{ass_A} and the three null-expectations), we deduce
\begin{equation*}
\begin{aligned}
\E_x\left[ \int_x^{Z_{t\wedge T_M}} \frac{\ud w}{g(w)-\theta w} \right]
 \leq -C(A)  \E[t\wedge T_M],
 \end{aligned}
\end{equation*}
 with $C(A)>0$. In other words, for any $x,t\geq 0$,
 \begin{equation*}
 \E_x\Big[t\wedge T_M\Big] \leq \frac{1}{C(A)} \E_x\left[ \int^x_{Z_{t\wedge T_M}} \frac{\ud w}{g(w)-\theta w} \right] \leq \frac{1}{C(A)} \int_{Me^{-1}}^{\infty} \frac{\ud w}{g(w)-\theta w}.
 \end{equation*}
 Hence using the Monotone Convergence Theorem,  as $t$ goes to $\infty$, we deduce \eqref{firststep}.

In order to  prove that the process becomes extinct almost surely, we first show that the time to extinction for the process $Z$ starting from $M$ is not almost surely infinite. We recall that we assumed that the environment  has no negative jumps larger than $1-e^{-1}$. Using  Proposition~\ref{comparison} (both processes with the same restriction on the negative jumps of the environment),  we observe that for any $x\leq M$, the process $(Z, \mathbb{P}_x)$ is stochastically dominated by $( Z^\sharp, \mathbb{P}^\sharp_x)$. The process $Z^\sharp$ is a CB-process in a L\'evy random environment (without competition) which is characterized by the branching mechanism  $\psi(\lambda)$. Since $\psi$ satisfies Grey's condition,  Theorem 4.1 of \cite{he2016continuous} ensures that there is $t_0>0$ for which 
 \[
 0<\P^\sharp_M\Big( Z^\sharp_{t_0}=0\Big)\leq \inf_{x \leq M}\P_x\Big(Z_{t_0}=0\Big):=p.
 \]
 Then we  denote by $T^x_{M}$ for the stopping time  $T_M$ under $\mathbb{P}_x$. Reasoning by recurrence and using the Markov property, we prove that  the extinction time of $Z$ is stochastically dominated from above by the random variable 
 $
 \sum_{i=1}^{\xi} (\bar\tau_i+t_0),
 $
 where $\xi$ is a geometric random variable that counts the number of random steps before $Z$ becomes extinct and $\{\bar\tau_i\}_{i\geq 0}$ are i.i.d., independent of $\xi$ and have the same distribution as $\sup_{x\geq 0}T^x_{M}$. To be more precise, the  algorithm is as follows: we start from $x$, we wait a random time $\tau_1\leq \bar\tau_1$ until the process is below the level $M$ and then the process becomes extinct before an time interval of size $t_0$ with probability $p$. If the process is not extinct after the time $\tau_1+t_0$, we start again the procedure thanks to the Markov property.
Hence,
 $$
 \sup_{x \geq 0}\E_x[T_0] \leq \frac{1}{p}\left({t_0}+\E\big[\sup_{x\geq 0} T_M^x\Big]\right)= \frac{1}{p}\left({t_0}+\sup_{x\geq 0} \E_x\big[T_M\Big]\right)<\infty,
 $$
 and in particular the process becomes extinct a.s.

 It remains to prove that the point $\infty$ is a continuous entrance point 
 and that the process can be extended  to a Feller process on $[0,\infty]$. Proposition~\ref{comparison} 
guarantees that the sequence of random variables $(T_m, \P_x)$ is increasing with respect to $x$. Thus,  it converges almost surely to a random variable here denoted by $T_m^{\infty}$. Then, from the first part of this proof, for any $m\geq M$,
\begin{equation}\label{eq_majTM}
\sup_{x\geq m}\E_x[T_m]\leq \int_{me^{-1}}^{\infty}\frac{\ud w}{g(w)-\theta w}\underset{m\to \infty}{\longrightarrow} 0.
\end{equation}
From Chebyshev's inequality and the Monotone Convergence Theorem, we deduce that for any $t>0$
\[
\lim_{m\to \infty}\lim_{x\to \infty}\P_x(T_m>t)\le \lim_{m\to \infty}\lim_{x\to \infty}\frac{E_x[T_m]}{t}=0.
\]
In other words the point $\infty$ is a continuous entrance point.

Next, we prove the extension to a Feller process on $[0,\infty]$. Let $\mathcal{C}_0([0,\infty])$ be the set of continuous functions on $[0,\infty]$ that vanish at $\infty$ and recall that  $(\mathtt{P}_t, t\geq 0)$ denotes the semigroup associated to the process $Z$ which is Feller on $[0,\infty)$. 

Let $t> 0$ be fixed. Recall, from Proposition~\ref{comparison}, that for any non-decreasing sequence $\{x_n\}_{n\ge 1}$ of strictly positive real numbers the sequence of  random variable $\{Z_t,\P_{x_n}\}_{n\ge 1}$ is non-decreasing. Hence it converges a.s. to a limit that we denote by $Z_t^\infty \in [0,\infty]$. Then, for any $f\in \mathcal{C}_0([0,\infty])$, $f$ is bounded and from the Dominated Convergence Theorem,  we deduce
$$
{\tt P}_t f(x)=\E_x[f(Z_t)]\underset{x\to \infty}{\longrightarrow}\E[f(Z_t^\infty)].
$$
We denote this limit by ${\tt P}_tf(\infty)$. Let us prove that the extension, defined as previously on $[0,\infty]$, gives a Feller semigroup. The definition through a limit guarantees that ${\tt P}$ remains a semigroup on $[0,\infty]$. Thus, according to Chapter III of \cite{RY}, it is sufficient to prove that for any $f\in \mathcal{C}_0([0,\infty])$,
\begin{equation*}
\lim_{t\to 0} \|{\tt P}_tf-f\|_{[0,\infty]}=0.
\end{equation*}
Observe that for any $t\geq 0$, $f\in \mathcal{C}_0([0,\infty])$, and $x>0$, we have
\begin{equation*}
\begin{aligned}
\|{\tt P}_tf-f\|_{[0,\infty]} & \leq \|{\tt P}_tf-f\|_{[0,\infty)} +|{\tt P}_tf(\infty)|\\
&\leq 2\|{\tt P}_tf-f\|_{[0,\infty)} +|{\tt P}_tf(\infty)-{\tt P}_tf(x)|+|f(x)|.
\end{aligned}
\end{equation*}
Since ${\tt P}$ is a Feller semigroup on $[0,\infty)$, we conclude that the  term of the right hand side of the previous inequality is small as soon as $t$ is sufficiently small and $x$ is chosen sufficiently large. 

 This ends the proof of Theorem~\ref{theo_meanfinite}.
 \end{proof}

\noindent \textbf{Acknowledgements.} Both authors  acknowledge support from  the Royal Society and CONACyT-MEXICO. This work was concluded whilst JCP was on sabbatical leave holding a David Parkin Visiting Professorship  at the University of Bath, he gratefully acknowledges the kind hospitality of the Department and University

\end{document}